\documentclass{article}
\usepackage{amsmath, amsthm, amsfonts}
\usepackage{amssymb}
\usepackage{mathrsfs}
\usepackage{euscript}
\newtheorem{theorem}{Theorem}[section]
\newtheorem{definition}{Definition}[section]

\numberwithin{equation}{section}

\title{Generalized Chebyshev polynomials of the second kind}
\author{Mohammad A. ALQUDAH
}

\begin{document}
\maketitle

\begin{abstract}
We characterize the generalized Chebyshev polynomials of the second kind (Chebyshev-II), then we provide a closed form of the generalized Chebyshev-II polynomials using Bernstein basis. These polynomials can be used to describe the approximation of continuous functions by Chebyshev
interpolation and Chebyshev series and how to compute efficiently such approximations. We conclude the paper with some results concerning  integrals of the generalized Chebyshev-II and Bernstein polynomials. 
\end{abstract}

\section{Introduction, background and motivation}
Orthogonal polynomials are very important and serve to approximate other functions, where the most commonly used orthogonal polynomials
are the classical orthogonal polynomials. The field of classical orthogonal polynomials developed in the late 19th century from a study
 of continued fractions by P.L. Chebyshev.

We have seen the significance of orthogonal polynomials, particularly in the solution of systems of linear equations and in the least-squares approximations. 
Meanwhile, polynomials can be represented in many different bases such as the monomial powers, Chebyshev, Bernstein, and Hermite bases forms.
Every form of polynomial basis has its advantages, and sometimes disadvantages. Many problems can be solved and many difficulties can be removed by suitable choice of basis.

In this paper we characterize the generalized Chebyshev polynomials of the second kind (Chebyshev-II) $\mathscr{U}_{r}^{(M,N)}(x).$ These polynomials can 
be used describe the approximation of continuous functions by Chebyshev interpolation and Chebyshev series and how to compute efficiently 
such approximations.

\subsection{Bernstein polynomials}
We recall a very concise overview of well-known results on Bernstein polynomials, followed by a brief summary of important properties.
\begin{definition}The $n+1$ Bernstein polynomials $B_{k}^{n}(x)$ of degree $n,$ $x\in [0,1], k=0,1,\ldots,n,$ are defined by:
\begin{equation}\label{binomial} B_{k}^{n}(x)=\left\{\begin{array}{ll}\binom {n}{k} x^{k}(1-x)^{n-k} & k=0,1,\dots,n \\0 & else\end{array}\right.,\end{equation}
where the binomial coefficients $$\binom{n}{k}=\frac{n!}{k!(n-k)!},\hspace{.1in}k=0,\dots,n.$$
\end{definition}

The Bernstein polynomials have been studied thoroughly and there are a fair amount of literature on these polynomials, they are known for their geometric properties \cite{Farin,Hoschek}, and the Bernstein basis form is known to be optimally stable.
They are all non-negative, $B_{k}^{n}(x)\geq 0,$ $x\in[0,1],$ satisfy symmetry relation $B_{k}^{n}(x)=B_{n-k}^{n}(1-x),$
and the product of two Bernstein polynomials is also a Bernstein polynomial which is given by
$$B_{i}^{n}(x)B_{j}^{m}(x)=\frac{\binom{n}{i}\binom{m}{j}}{\binom{n+m}{i+j}}B_{i+j}^{n+m}(x).$$

The Bernstein polynomials of degree $n$ can be defined by combining two Bernstein polynomials
of degree $n-1,$ where the $k$th $n$th-degree Bernstein polynomial can be written by the known recurrence relation as
$$B_{k}^{n}(x)=(1-x)B_{k}^{n-1}(x)+xB_{k-1}^{n-1}(x), \hspace{.1in} k=0,\dots,n; n\geq 1$$ where $B_{0}^{0}(x)=0$ and $B_{k}^{n}(x)=0$ for $k<0$ or $k>n.$
Moreover, it is possible to write each Bernstein polynomials of degree $r$ where $r\leq n$ in terms of Bernstein polynomials of degree $n$ using the following degree elevation \cite{Farouki3}:
\begin{equation}\label{ber-elv} B_{k}^{r}(x)=\sum_{i=k}^{n-r+k}\frac{\binom{r}{k}\binom{n-r}{i-k}}{\binom{n}{i}}B_{i}^{n}(x),\hspace{.1in}k=0,1,\dots,r.
\end{equation}
In addition, the Bernstein polynomials can be differentiated and integrated easily as 
$$\frac{d}{dx}B_{k}^{n}(x)=n[B_{k-1}^{n-1}(x)-B_{k}^{n-1}(x)],\hspace{.1in}n\geq 1,$$
and
\begin{equation}\label{bern-intg}\int_{0}^{1}B_{k}^{n}(x)dx=\frac{1}{n+1},\hspace{.1in} k=0,1,\dots,n.\end{equation}
 
These analytic and geometric properties of Bernstein polynomials with the advent of computer graphics made Bernstein polynomials important in the
 form of B\'{e}zier curves and B\'{e}zier surfaces in Computer Aided Geometric Design (CAGD). The Bernstein polynomials are the standard basis for the B\'{e}zier  representations of curves and surfaces in CAGD.

However, the Bernstein polynomials are not orthogonal and could not be used effectively in the least-squares approximation \cite{Rice}. Since then a theory
 of approximation has been developed and many approximation methods have been introduced and analyzed. The method of least squares approximation accompanied
 by orthogonal polynomials is one of these approximation methods.

\subsection{Least-square approximation}
The idea of least squares can be applied to approximating a given function by a weighted sum of other functions. The best approximation  can be defined as that which minimizes the difference between the original function and the approximation; for a least-squares approach the quality  of the approximation is measured in terms of the squared differences between the two. The following will briefly refresh our background information to enable us to combine the superior least square performance of the generalized Chebyshev-II polynomials with the geometric insight of the Bernstein form.
\begin{definition}
For a function $f(x),$ continuous on $[0,1]$ the least square approximation requires finding a polynomial (Least-Squares Polynomial)
$$p_{n}(x)=a_{0}\varphi_{0}(x)+a_{1}\varphi_{1}(x)+\dots+a_{n}\varphi_{n}(x)$$ 
that minimize the error
$$E(a_{0},a_{1},\dots,a_{n})=\int_{0}^{1}[f(x)-p_{n}(x)]^{2}dx.$$
\end{definition}

For minimization, the partial derivatives must satisfy
$$\frac{\partial E}{\partial a_{i}}=0, i=0,\dots,n.$$
These conditions will give rise to a system of $n+1$ normal equations in $n+1$ unknowns: $a_{0},a_{1},\dots,a_{n}.$ Solution of these equations
 will yield the unknowns: $a_{0},a_{1},\dots,a_{n}$ of the least-squares polynomial $p_{n}(x).$ 
It is important to choose a suitable basis, for example by choosing $\varphi_{i}(x)=x^{i},$ the matrix coefficients of the normal equations given as
$$(H_{n+1}(0,1))_{i,k}=\int_{0}^{1}x^{i+1}dx, \hspace{.05in} 0\leq i,k\leq n,$$
which is Hilbert matrix that has round-off error difficulties and notoriously ill-conditioned for even modest values of $n.$

However, the computations can be made efficient by using orthogonal polynomials.  Choosing $\{\varphi_{0}(x),\varphi_{1}(x),\dots,\varphi_{n}(x)\}$
 to be orthogonal simplifies the least-squares approximation procedures. The matrix of the normal equations will be diagonal which simplifies calculations
 and gives a compact closed form for $a_{i}, i=0,1,\dots,n.$ 

Moreover, knowing $p_{n}(x)$ will be sufficient to compute $a_{n+1}$ to get  $p_{n+1}(x).$  See \cite{Rice} for more details on the least squares  approximations.

\subsection{Factorial minus half}
We present some results concerning factorials, double factorials and some combinatorial identities.
The double factorial of an integer $n$ is given by
\begin{equation}\label{douvle-factrrial}
\begin{aligned}
(2n-1)!!&=(2n-1)(2n-3)(2n-5)\dots(3)(1) \hspace{.2in}  \text{if $n$ is odd} \\
n!!&=(n)(n-2)(n-4)\dots(4)(2) \hspace{.63in} \text{if $n$ is even},
\end{aligned}
\end{equation}
where $0!!=(-1)!!=1.$

From \eqref{douvle-factrrial}, we can derive the following relation for factorials
\begin{equation}\label{doublefac} n!!=\left\{\begin{array}{ll}2^{\frac{n}{2}}(\frac{n}{2})! & \text{if $n$ is even} \\
\frac{n!}{2^{\frac{n-1}{2}}(\frac{n-1}{2})!} & \text{if $n$ is odd} \end{array}\right..\end{equation}
From \eqref{doublefac} we obtain
\begin{equation}\label{eq1}(2n)!!=[2(n)][2(n-1)]\dots[2.1]=2^{n}n!,\end{equation}
and
\begin{equation}\label{eq2}
(2n)!
=[(2n-1)(2n-3)\dots(1)]\left([2(n)][2(n-1)][2(n-2)]\dots [2(1)]\right)=(2n-1)!!2^{n}n!.
\end{equation}

By combining \eqref{eq1} and \eqref{eq2}, we get 
\begin{equation}\label{half-double}
\binom{2n}{n}=\frac{2^{2n}(2n-1)!!}{(2n)!!}.
\end{equation}
In addition, using \eqref{eq2} and with some simplifications we obtain
\begin{equation}\label{2nd-eq}
\frac{\binom{2n}{2k}}{\binom{n}{k}}=\frac{(2n-1)!!}{(2k-1)!!(2n-2k-1)!!}.
\end{equation}

\subsection{Univariate Chebyshev-II polynomials}
Let us first consider a definition and some properties of the univariate Chebyshev polynomials of the second kind.
\begin{definition}(Chebyshev polynomials of the second kind $U_{n}(x)$).
The Chebyshev polynomial of the second kind of order $n$ is defined as follows:
\begin{equation}
U_{n}(x)=\frac{\sin[(n+1)\cos^{-1}(x)]}{sin[\cos^{-1}(x)]}, x\in[-1,1],\hspace{.1in} n=0,1,2,\dots.
\end{equation}
From this definition the following property is evident:
\begin{equation}
U_{n}(x)=\frac{\sin(n+1)\theta}{sin\theta}, \hspace{.1in} x=\cos\theta.
\end{equation}
\end{definition}
The Chebyshev polynomials are special cases of Jacobi polynomials $P_{n}^{(\alpha,\beta)}(x),$ and related as
\begin{equation}\label{Tsch-Jac-rel}U_{n}(x)=(n+1)\binom{n+\frac{1}{2}}{n}^{-1}P_{n}^{(\frac{1}{2},\frac{1}{2})}(x).\end{equation}

Authors are not uniform in orthogonal polynomials notations,  and for the convenience we recall the following explicit expressions
for univariate Chebyshev-II polynomials of degree $n$ in $x$, see Szeg\"{o} \cite{Szego}:
\begin{equation}
U_{n}(x):=\frac{(n+1)(2n)!!}{(2n+1)!!}\sum_{k=0}^{n}\binom{n+\frac{1}{2}}{n-k}\binom{n+\frac{1}{2}}{k}\left(\frac{x+1}{2}\right)^{n-k}\left(\frac{x-1}{2}\right)^{k},
\end{equation}
which it can be transformed in terms of Bernstein basis on $x\in[0,1]$,
\begin{equation}\label{Chebyshev-II-basis fromat}
U_{n}(2x-1):=\frac{(n+1)(2n)!!}{(2n+1)!!}\sum_{k=0}^{n}(-1)^{n+1}\frac{\binom{n+\frac{1}{2}}{k}\binom{n+\frac{1}{2}}{n-k}}{\binom{n}{k}}B_{k}^{n}(x).
\end{equation}
Also, it may be represented in terms of Gauss hypergeometric series as follows \cite{Olver}: 
\begin{equation}\label{hypergeo}
U_{n}(x):=(n+1) {}_{2}F_{1}\left(\begin{matrix}-n,n+2\\ \frac{3}{2}\end{matrix}; \frac{1-x}{2}\right).
\end{equation}
Although the Pochhammer symbol is more appropriate, but the combinatorial notation gives more compact and clear formulas, these have also been used by Szeg\"{o} \cite{Szego}.

\subsubsection{Properties of the Chebyshev-II polynomials $U_{n}(x)$}
The Chebyshev-II polynomials $U_{n}(x)$ of degree $n$ are orthogonal polynomials, except for a constant factor, with respect to the weight function
$$\mathrm{W}(x)=\sqrt{1-x^{2}}.$$
In addition, Chebyshev-II polynomials satisfy the orthogonality relation \cite{Gradshtein}
\begin{equation}\label{ortho-rel}
\int_{0}^{1}x^{\frac{1}{2}}(1-x)^{\frac{1}{2}}U_{n}(x)U_{m}(x)dx=\left\{\begin{array}{ll} 
0             & \mbox{if } m\neq n \\
\frac{\pi}{8} & \mbox{if } m=n
\end{array}\right..
\end{equation}
The univariate classical orthogonal polynomials are traditional defined on $[-1,1],$ however, it is more convenient to use $[0,1].$  

\section{Main results}
In this section, we characterize the generalized Chebyshev-II polynomials $\mathscr{U}_{n}^{(M,N)}(x),$  then we write them using Bernstein basis.
Finally, we conclude the section with the explicit formula for the generalized Chebyshev-II polynomials using Bernstein basis.
\subsection{Characterization}
Using the relation \eqref{Tsch-Jac-rel} and a construction similar to \cite{Bav-Koek2,Koornwinder} for $M, N\geq 0,$ the generalized Chebyshev-II polynomials
$\left\{\mathscr{U}_{n}^{(M,N)}(x)\right\}_{n=0}^{\infty}$ can be written as
\begin{equation}\label{gen-Chebyshev-II}
\mathscr{U}_{n}^{(M,N)}(x)=\frac{(2n+1)!!}{2^{n}(n+1)!}U_{n}(x)+MQ_{n}(x)+NR_{n}(x)+MNS_{n}(x),\hspace{.05in}n=0,1,2,\dots
\end{equation}
where for $n=1,2,3,\dots$
\begin{equation}\label{q-fn}
Q_{n}(x)=
\frac{(2n+1)!!}{3.2^{n}n!}\left[n(n+2)U_{n}(x)-\frac{3}{2}(x-1)DU_{n}(x)\right],
\end{equation}
\begin{equation}\label{r-fn}
R_{n}(x)=\frac{(2n+1)!!}{3.2^{n}n!}\left[n(n+2)U_{n}(x)-\frac{3}{2}(x+1)DU_{n}(x)\right],
\end{equation}
and
\begin{equation}\label{s-fn}
S_{n}(x)=\frac{(n+2)!(2n+1)!!}{3^{2}.2^{n}n!(n-1)!}[n(n+2)U_{n}(x)-3x DU_{n}(x)].
\end{equation}
If we use
\begin{equation}
(x^{2}-1)D^{2}U_{n}(x)=n(n+2)U_{n}(x)-3x DU_{n}(x),\hspace{.05in} n=0,1,2,3,\dots
\end{equation}
we easily find from \eqref{s-fn} that
\begin{equation}\label{s-fn-useJac}
S_{n}(x)=\frac{(n+2)!(2n+1)!!}{3^{2}.2^{n}n!(n-1)!}(x^{2}-1)D^{2}U_{n}(x), \hspace{.1in}n=1,2,3,\dots.
\end{equation}
It is clear that $Q_{0}(x)=R_{0}(x)=S_{0}(x)=0.$

We know that the generalized Chebyshev-II polynomials satisfy the symmetry relation \cite{Koornwinder},
\begin{equation}
\mathscr{U}_{n}^{(M,N)}(x)=(-1)^{n}\mathscr{U}_{n}^{(N,M)}(-x),\hspace{.1in} n=0,1,2,\dots
\end{equation}
which implies that $Q_{n}(x)=(-1)^{n}R_{n}(-x), S_{n}(x)=(-1)^{n}S_{n}(-x)$ for $n=0,1,\dots.$

From \eqref{q-fn} and \eqref{r-fn} if follows that
\begin{equation}
Q_{n}(1)=
\frac{(n+2)(2n+1)!!}{3.2^{n}(n-1)!}U_{n}(1),\hspace{.1in} n=1,2,3,\dots
\end{equation}
\begin{equation}
R_{n}(-1)=\frac{(n+2)(2n+1)!!}{3.2^{n}(n-1)!}U_{n}(-1),\hspace{.1in} n=1,2,3,\dots
\end{equation}

Note that the representations \eqref{q-fn} and \eqref{r-fn} imply that for $n=1,2,3,\dots,$ we have 
\begin{equation}\label{q-fn-sum}
Q_{n}(x)=\sum_{k=0}^{n}q_{k}\frac{(2k+1)!!}{2^{k}(k+1)!}U_{k}(x) \hspace{.15in} \text{with} \hspace{.15in} 
q_{n}=\frac{n(n+1)(2n+1)}{6}
\end{equation}
and
\begin{equation}\label{r-fn-sum}
R_{n}(x)=\sum_{k=0}^{n}r_{k}\frac{(2k+1)!!}{2^{k}(k+1)!}U_{k}(x)\hspace{.15in} \text{with} \hspace{.15in} 
r_{n}=\frac{n(n+1)(2n+1)}{6}.
\end{equation}
We also can find from \eqref{s-fn} that for $n=1,2,3,\dots,$ we have
\begin{equation}\label{s-fn-sum}
S_{n}(x)=\sum_{k=0}^{n}s_{k}\frac{(2k+1)!!}{2^{k}(k+1)!}U_{k}(x)\hspace{.15in} \text{with} \hspace{.15in}
s_{n}=\frac{(n+2)(n+1)^{2}n^{2}(n-1)}{9}.
\end{equation}

Therefore, for $M, N\geq 0,$ the generalized Chebyshev-II polynomials $\left\{\mathscr{U}_{n}^{(M,N)}(x)\right\}_{n=0}^{\infty}$ are orthogonal on the interval $[-1,1]$ with respect to the weight function
\begin{equation}
\frac{2}{\pi}(1-x)^{\frac{1}{2}}(1+x)^{\frac{1}{2}}+M\delta(x+1)+N\delta(x-1).
\end{equation}
and can be written as
\begin{equation}\label{gen-Chebyshev-II-ccomb}
\mathscr{U}_{n}^{(M,N)}(x)=\frac{(2n+1)!!}{2^{n}(n+1)!}U_{n}(x)+\sum_{k=0}^{n} \lambda_{k} \frac{(2k+1)!!}{2^{k}(k+1)!}U_{k}(x),
\end{equation}
where
\begin{equation}\label{lmda}\lambda_{k}=M q_{k}+N r_{k}+M N  s_{k}.\end{equation}

The next theorem provides a closed form for generalized Chebyshev-II polynomial $\mathscr{U}_{r}^{(M,N)}(x)$ of degree $r$ as
a linear combination of the Bernstein polynomials $B_{i}^{r}(x), i=0,1,\dots,r$ of degree $r.$ Those results  generalize \cite{rababah4} contributions to the univariate Chebyshev polynomials of the second kind. 
\begin{theorem}\label{gen-jacinBer form}
For $M,N\geq 0,$ the generalized Chebyshev-II polynomials $\mathscr{U}_{r}^{(M,N)}(x)$ of degree $r$ have the following
Bernstein representation:
\begin{equation}\label{gen-jac-inBer-r}
\mathscr{U}_{r}^{(M,N)}(x)=\frac{(2r+1)!!}{2^{r}(r+1)!}\sum\limits_{i=0}^{r}(-1)^{r-i}\vartheta_{i,r}B_{i}^{r}(x)+\sum_{k=0}^{r} \lambda_{k}\frac{(2k+1)!!}{2^{k}(k+1)!} \sum\limits_{i=0}^{k}(-1)^{k-i}\vartheta_{i,k}B_{i}^{k}(x)
\end{equation}
where $\lambda_{k}=M q_{k}+N r_{k}+M N  s_{k}$
and
$$\vartheta_{i,r}=\frac{(2r+1)^{2}}{2^{2r}(2r-2i+1)(2i+1)}\frac{\binom{2r}{r}\binom{2r}{2i}}{\binom{r}{i}}, \hspace{.05in}i=0,1,\dots,r,$$
 where $$\vartheta_{0,r}=\frac{(2r+1)}{2^{2r}}\binom{2r}{r}.$$

The coefficients $\vartheta_{i,r}$ satisfy the recurrence relation
\begin{equation}\label{recur}\vartheta_{i,r}=\frac{(2r-2i+3)}{(2i+1)}\vartheta_{i-1,r}, \hspace{.1in}i=1,\dots,r.\end{equation}
\end{theorem}
\begin{proof}
To write a generalized Chebyshev-II polynomial $\mathscr{U}_{r}^{(M,N)}(x)$ of degree $r$ as a linear combination of the Bernstein polynomial basis $B_{i}^{r}(x), i=0,1,\dots,r$ of degree $r$ in explicit form, we begin with substituting 
\eqref{binomial} into \eqref{gen-Chebyshev-II-ccomb} to get
\begin{equation}
\begin{aligned}
\mathscr{U}_{r}^{(M,N)}(x)&=\frac{(2r+1)!!}{2^{r}(r+1)!}\sum_{i=0}^{r}\frac{\binom{r+\frac{1}{2}}{r-i}\binom{r+\frac{1}{2}}{i}}{\binom{r}{r-i}}B_{r-i}^{r}(x)+\sum_{k=0}^{r}\lambda_{k}\frac{(2k+1)!!}{2^{k}(k+1)!}
\sum_{j=0}^{k}\frac{\binom{k+\frac{1}{2}}{k-j}\binom{k+\frac{1}{2}}{j}}{\binom{k}{k-j}}B_{k-j}^{k}(x)\\
&=\frac{(2r+1)!!}{2^{r}(r+1)!}\sum_{i=0}^{r}(-1)^{r-i}
\vartheta_{i,r}
B_{i}^{r}(x)+\sum_{k=0}^{r}\lambda_{k}\frac{(2k+1)!!}{2^{k}(k+1)!}
\sum_{j=0}^{k}(-1)^{k-j}
\vartheta_{j,k}
B_{j}^{k}(x),
\end{aligned}
\end{equation}
where
\begin{equation}\label{theta}\vartheta_{i,r}=\frac{\binom{r+\frac{1}{2}}{i}\binom{r+\frac{1}{2}}{r-i}}{\binom{r}{i}}, i=0,1,\dots,r.\end{equation}

Using \eqref{theta} and applying 
$\left(n+\frac{1}{2}\right)!=\frac{\sqrt{\pi}}{2^{n+1}}(2n+1)!!$
with some simplifications, we have
\begin{align*}
\binom{r+\frac{1}{2}}{i}\binom{r+\frac{1}{2}}{r-i}
&=\frac{(2r+1)\left(r-\frac{1}{2}\right)!}{(2r-2i+1)(r-i)!(i-\frac{1}{2})!}\frac{(2r+1)\left(r-\frac{1}{2}\right)!}{(2i+1)i!(r-i-\frac{1}{2})!}\\
&=\frac{2^{i}(2r+1)(r-1)!!}{2^{r}(2r-2i+1)(r-i)!(i-1)!!}\frac{2^{r-i}(2r+1)(2r-1)!!}{2^{r}i!(2i+1)(2r-2i-1)!!}\\
&=\frac{(2r+1)}{2^{r}(2r-2i+1)(r-i)!i!}\frac{(2r-1)!!}{(2i-1)!!}\frac{(2r+1)}{(2i+1)}\frac{(2r-1)!!}{(2(r-i)-1)!!}.
\end{align*}
Using the fact $(2n)!=(2n-1)!!2^{n}n!$ we get 
$$\binom{r+\frac{1}{2}}{r-i}\binom{r+\frac{1}{2}}{i}=\frac{(2r+1)^{2}}{2^{2r}(2r-2i+1)(2i+1)}\binom{2r}{r}\binom{2r}{2i}.$$

For the recurrence relation, it is clear that for $i=1,\dots,r$ we have
\begin{equation*}\vartheta_{i-1,r}=\frac{(i+\frac{1}{2})\binom{r+\frac{1}{2}}{i}\binom{r+\frac{1}{2}}{r-i}}{(r-i+\frac{3}{2})\binom{r}{i}}.\end{equation*}
Thus,
\begin{equation}\vartheta_{i-1,r}=\frac{(2i+1)(2r+1)^{2}\binom{2r}{r}\binom{2r}{2i}}{2^{2r}(2r-2i+1)(2i+1)(2r-2i+3)\binom{r}{i}}, \hspace{.1in}i=1,\dots,r.\end{equation}
\end{proof}

We conclude with an interesting integration property of the weighted generalized
Chebyshev-II polynomials with the Bernstein polynomials. To do this, we introduce the following definition.
\begin{definition}
The Eulerian integral of the first kind is a function of two complex variables defined by
$$B(x,y)=\int_{0}^{1}u^{x-1}(1-u)^{y-1}du,\hspace{.5in}\Re(x),\Re(y)>0.$$
\end{definition}
Note that the Eulerian integral of the first kind  is often called Beta integral.  We observe that the beta integral is symmetric, a change
of variables by $t=1-u$ clearly illustrates this.
\begin{theorem}\label{gen-int-ber-jac}
Let $B_{r}^{n}(x)$ be the Bernstein polynomial of degree $n$ and $\mathscr{U}_{i}^{(M,N)}(x)$ be the generalized Chebyshev-II polynomial of degree $i,$ then
for $i,r=0,1,\dots,n$ we have
\begin{equation}
\begin{aligned}
&\int_{0}^{1}x^{\frac{1}{2}}(1-x)^{\frac{1}{2}}B_{r}^{n}(x)\mathscr{U}_{i}^{(M,N)}(x)dx\\
&=\binom{n}{r}
\frac{(2i+1)!!}{2^{i}(i+1)!}
\sum_{k=0}^{i}(-1)^{i-k}\binom{i+\frac{1}{2}}{k}\binom{i+\frac{1}{2}}{i-k}B(r+k+\frac{3}{2},n+i-r-k+\frac{3}{2})\\
&+\sum_{d=0}^{i}\lambda_{d}\binom{n}{r}
\frac{(2d+1)!!}{2^{d}(d+1)!}
\sum_{j=0}^{d}(-1)^{d-j}\binom{d+\frac{1}{2}}{j}\binom{d+\frac{1}{2}}{d-j}B(r+j+\frac{3}{2},n+d-r-j+\frac{3}{2})
\end{aligned}
\end{equation}
where $B(x,y)$ is the beta function.
\end{theorem}
\begin{proof}
By using \eqref{gen-jac-inBer-r}, the integral $$I=\int_{0}^{1}x^{\frac{1}{2}}(1-x)^{\frac{1}{2}}B_{r}^{n}(x)\mathscr{U}_{i}^{(M,N)}(x)dx,$$ can be simplified to
\begin{equation}\begin{aligned}
I&=\frac{(2i+1)!!}{2^{i}(i+1)!}\int_{0}^{1}x^{r+\frac{1}{2}}(1-x)^{n-r+\frac{1}{2}}\binom{n}{r}\sum_{k=0}^{i}(-1)^{i-k}\frac{\binom{i+\frac{1}{2}}{k}\binom{i+\frac{1}{2}}{i-k}}{\binom{i}{k}}B_{k}^{i}(x)\\
&+\sum_{d=0}^{i}\lambda_{d}\frac{(2d+1)!!}{2^{d}(d+1)!}\int_{0}^{1}x^{r+\frac{1}{2}}(1-x)^{n-r+\frac{1}{2}}\binom{n}{r}\sum_{j=0}^{d}(-1)^{d-j}\frac{\binom{d+\frac{1}{2}}{j}\binom{d+\frac{1}{2}}{d-j}}{\binom{d}{j}}B_{j}^{d}(x)dx\\
&=\frac{(2i+1)!!}{2^{i}(i+1)!}\binom{n}{r}\sum_{k=0}^{i}(-1)^{i-k}\binom{i+\frac{1}{2}}{k}\binom{i+\frac{1}{2}}{i-k}\int_{0}^{1}x^{r+k+\frac{1}{2}}(1-x)^{n+i-r-k+\frac{1}{2}}dx\\
&+\sum_{d=0}^{i}\lambda_{d}\frac{(2d+1)!!}{2^{d}(d+1)!}\binom{n}{r}\sum_{j=0}^{d}(-1)^{d-j}\binom{d+\frac{1}{2}}{j}\binom{d+\frac{1}{2}}{d-j}\int_{0}^{1}x^{r+j+\frac{1}{2}}(1-x)^{n+d-r-j+\frac{1}{2}}dx
\end{aligned}
\end{equation}
The integrals in the last equation are the Eulerian integral of the first kind, and can be written in term of Beta function as $B(x_{i},y_{i})$ with $x_{1}=r+k+\frac{3}{2},$ $y_{1}=n+i-r-k+\frac{3}{2},$
$x_{2}=r+j+\frac{3}{2},$ and $y_{2}=n+d-r-j+\frac{3}{2}.$
\end{proof}

\section*{Acknowledgment}
The author would like to thank the editor and anonymous reviewers for their valuable comments and suggestions, which were helpful in improving the paper. 


\end{document}